\documentclass{birkjour}
 \newtheorem{thm}{Theorem}
 \newtheorem{cor}[thm]{Corollary}
 \newtheorem{lem}[thm]{Lemma}
 
 \theoremstyle{definition}
 \newtheorem{defn}[thm]{Definition}

 \newtheorem{rem}[thm]{Remark}

\usepackage{amsmath}
\usepackage{amssymb}
\usepackage{dsfont}
 \usepackage[mathscr]{eucal}
\usepackage[english]{babel}
 \numberwithin{equation}{section}
 \numberwithin{thm}{section}
\newcommand{\ccomma}{\mathpunct{\raisebox{0.5ex}{,}}}
\newcommand{\dist}{\operatorname{dist}}

\def\leq{\leqslant}

\def\geq{\geqslant}

\begin{document}

\title[A functional model for the Fourier--Plancherel operator]{A functional model for the Fourier--Plancherel operator
truncated to the positive semiaxis}

\author{V.~Katsnelson}

\address{Department of Mathematics,\\
The Weizmann Institute,\\
76100, Rehovot, Israel}

\email{victor.katsnelson@weizmann.ac.il;\\ victorkatsnelson@gmail.com}

\date{27 октября, 2017}
\dedicatory{This paper is dedicated to the memory of my colleague 
Mikhail Solomyak.}
\begin{abstract}
The truncated Fourier operator
\(\mathscr{F}_{\mathbb{R^{+}}}\),
\begin{equation*}%
(\mathscr{F}_{\mathbb{R^{+}}}x)(t)=\frac{1}{\sqrt{2\pi}}%
\int\limits_{\mathbb{R^{+}}}x(\xi)e^{it\xi}\,d\xi\,,\ \ \
t\in{}{\mathbb{R^{+}}},
\end{equation*}%
is studied. The operator \(\mathscr{F}_{\mathbb{R^{+}}}\) is viewed as an
operator acting in the spa\-ce~\(L^2(\mathbb{R^{+}})\). A functional model
for the operator \(\mathscr{F}_{\mathbb{R^{+}}}\) is constructed. This \break functional
model is the operator of multiplication by an appropriate\break (${2\times2}$)-ma\-t\-rix function
acting in the space \(L^2(\mathbb{R^{+}})\oplus{}L^2(\mathbb{R^{+}})\). Using this functional model,
the spectrum of the operator \(\mathscr{F}_{\mathbb{R^{+}}}\) is found. The resolvent
of the operator \(\mathscr{F}_{\mathbb{R^{+}}}\) is estimated near its spectrum.
\end{abstract}
\maketitle

\noindent
\textsf{Notation:}\\
\(\mathbb{R}\) \ \ is  the set of all real numbers.\\
\(\mathbb{R}^{+}\) is the set of all positive real numbers.\\
\(\mathbb{C}\) \ \ is the set of all complex numbers.\\
\(\mathbb{Z}\) \ \ is the set of all integes.\\
\(\mathbb{N}=\{1,2,3,\,\dots\}\) is the set of all natural numbers.
\setcounter{section}{0}
\section{The Fourier--Plancherel operator truncated to the positive semiaxis.}

In this paper we study the truncated Fourier operator
\(\mathscr{F}_{\mathbb{R^{+}}}\),
\begin{equation}%
\label{DTFTr}
(\mathscr{F}_{\mathbb{R^{+}}}x)(t)=\frac{1}{\sqrt{2\pi}}%
\int\limits_{\mathbb{R^{+}}}x(\xi)e^{it\xi}\,d\xi\,,\ \ \
t\in{}{\mathbb{R^{+}}}.
\end{equation}%
The operator \(\mathscr{F}_{\mathbb{R^{+}}}\) is viewed as an
operator acting in the space \(L^2(\mathbb{R^{+}})\) of all square integrable
 complex-valued functions on \(\mathbb{R^{+}}\) equipped
with the inner product
\begin{equation*}
\langle{}x,y\rangle_{L^2(\mathbb{R^{+}})}=\int\limits_{\mathbb{R^{+}}}x(t)\overline{y(t)}\,dt.
\end{equation*}
The operator \(\mathscr{F}^{ \ast}_{\mathbb{R^{+}}}\) adjoint to~\(\mathscr{F}_{\mathbb{R^{+}}}\) with respect
 to this inner product is
\begin{equation}%
\label{DTFTrA}
(\mathscr{F}_{\mathbb{R^{+}}}^{\ast}x)(t)=\frac{1}{\sqrt{2\pi}}%
\int\limits_{\mathbb{R^{+}}}x(\xi)e^{-it\xi}\,d\xi\,,\ \ \
t\in{}{\mathbb{R^{+}}}.
\end{equation}%
The operator \(\mathscr{F}_{\mathbb{R^{+}}}\) has the form
\begin{equation}%
\label{UnDil}%
 \mathscr{F}_{\mathbb{R^{+}}}=
P_{\mathbb{R^{+}}}\,\mathscr{F}\,P_{\mathbb{R^{+}}|_{{\scriptstyle
L}^2{\scriptstyle({\mathbb{R}}^{+})}}}\,,
\end{equation}%
where \(\mathscr{F}\) is the Fourier--Plancherel operator on the whole real
axis:
\begin{equation}
\label{FWRA}
(\mathscr{F}x)(t)=\frac{1}{\sqrt{2\pi}}%
\int\limits_{\mathbb{R}}x(\xi)e^{it\xi}\,d\xi\,,\ \ \
t\in{}{\mathbb{R}},
\end{equation}
\begin{equation*}
\mathscr{F}:\,L^2(\mathbb{R})\to{}L^2(\mathbb{R})\,,
\end{equation*}
and \(P_{_{\scriptstyle\mathbb{R^{+}}}}\) is the natural
orthogonal projector from \(L^2(\mathbb{R})\) onto
\(L^2(\mathbb{R}^{+})\):
\begin{equation}
\label{NaPr}%
 (P_{\mathbb{R^{+}}}x)(t)=\mathds{1}_{_{\scriptstyle\mathbb{R}^{+}}}\!(t)\,x(t),\ \
 x\in{}L^2(\mathbb{R}),\quad t\in\mathbb{R}^{+},
\end{equation}
where \(\mathds{1}_{_{\scriptstyle\mathbb{R^{+}}}}\!(t)\) is the indicator
function of the set \(\mathbb{R}^{+}\). For any set \(E\), its
indicator function \(\mathds{1}_{_{\scriptstyle E}}\) is
\begin{equation}
\label{IndF}
\mathds{1}_E(t)= \begin{cases}1&\ \textup{if}\ t\in{}E,\\
0&\ \textup{if}\ t\not\in{}E.
\end{cases}
\end{equation}
It should be mentioned that the Fourier operator \(\mathscr{F}\) is a unitary operator in~\(L^2(\mathbb{R})\):
\begin{equation}
\label{Unit}
\mathscr{F}^{\ast}\mathscr{F}=\mathscr{F}\mathscr{F}^{\ast}=
\mathscr{I}_{L^2(\mathbb{R})},
\end{equation}
where \(\mathscr{I}_{L^2(\mathbb{R})}\) is the identity operator in
\(L^2(\mathbb{R})\) and \(\mathscr{F}^{\ast}\) is the operator
adjoint to~\(\mathscr{F}\) with respect to the
standard inner product in \(L^2(\mathbb{R})\).

From \eqref{UnDil} and  \eqref{Unit} it follows that the operators
\(\mathscr{F}_{\mathbb{R^{+}}}\) and
\(\mathscr{F}_{\mathbb{R^{+}}}^{\ast}\) are contractive: %
\(\|\mathscr{F}_{\mathbb{R^{+}}}\|\leq1\),
\(\|\mathscr{F}_{\mathbb{R^{+}}}^{\ast}\|\leq1\).  Later it will be shown that actually
\begin{equation}
\label{NorTrF}%
 \|\mathscr{F}_{\mathbb{R^{+}}}\|=1,\quad
\|\mathscr{F}_{\mathbb{R^{+}}}^{\ast}\|=1.
\end{equation}
Nevertheless, these operators are strictly contractive:
\begin{equation}
\label{StrCon}%
 \|\mathscr{F}_{\!_{\scriptstyle\mathbb{R_{+}}}}x\|<\|x\|,\quad
 \|\mathscr{F}^{\ast}_{\!_{\scriptstyle\mathbb{R^{+}}}}x\|<\|x\|\, \quad
 \mbox{for all}\quad x\in{}L^2(\mathbb{R}^{+}),\quad x\not=0\,,
\end{equation}
and their spectral radii
\(r(\mathscr{F}_{\!_{\scriptstyle\mathbb{R^{+}}}})\) and
\(r(\mathscr{F}^{\ast}_{\!_{\scriptstyle\mathbb{R^{+}}}})\) are
less than one:
\begin{equation}
\label{SpRad}%
r(\mathscr{F}_{\!_{\scriptstyle\mathbb{R^{+}}}})=
r(\mathscr{F}^{\ast}_{\!_{\scriptstyle\mathbb{R^{+}}}})=1/\sqrt{2}.
\end{equation}
In particular, the operators
\(\mathscr{F}_{\!_{\scriptstyle\mathbb{R^{+}}}}\) and
\(\mathscr{F}^{\ast}_{\!_{\scriptstyle\mathbb{R^{+}}}}\) are
contractions of class~\(C_{00}\) in the sense of \cite{Sz}.
(See \cite[Chapter 2, \S4]{Sz}.)

In \cite{Sz}, the spectral theory of contractions in  Hilbert
space was developed. The starting point of this theory is the
representation of a given contractive operator \(A\) acting in
a Hilbert space \(\mathscr{H}\) in the form
\begin{equation}
\label{ComCon}%
 A=PUP,
\end{equation}
 where \(U\) is a unitary operator acting is some \emph{ambient} Hilbert
space \(\mathfrak{H}\), \(\mathscr{H}\subset\mathfrak{H}\), and
\(P\) is the orthogonal projector from \(\mathfrak{H}\) onto \(\mathscr{H}\). In the
construction of \cite{Sz} it is required that not only~\eqref{ComCon} but also the series of
identities
\begin{equation}
\label{Dilat}%
 A^n=PU^nP,\quad n\in\mathbb{N},
\end{equation}
be true. The unitary operator \(U\) acting in an ambient Hilbert
space \(\mathfrak{H}\), \(\mathcal{H}\subset\mathfrak{H}\), is
called  a \emph{unitary dilation of the operator \(A\),}
\(A:\mathcal{H}\to\mathcal{H}\), if identities \eqref{Dilat}
are fulfilled. In  \cite{Sz} it was shown that every contractive
operator \(A\) admits a unitary dilation. By using the unitary
dilation, a functional model of the operator \(A\) was constructed.
This functional model is an operator acting in some Hilbert space
of analytic functions. The functional model of the operator \(A\)
is an operator  unitarily equivalent to \(A\). The spectral
theory of the original operator \(A\) is developed by analyzing
its functional model.
However, the functional model constructed in \cite{Sz} is not suitable for the spectral analysis of the truncated
Fourier--Plancherel operator \(\mathscr{F}_{\mathbb{R^{+}}}\).

Relation \eqref{UnDil} is of the form \eqref{ComCon}, where
\(\mathscr{H}=L^2(\mathbb{R}^{+})\),
\(\mathfrak{H}=L^2(\mathbb{R})\),\break
 ${U\!=\!\mathscr{F}}$, $A=\mathscr{F}_{_{\scriptstyle\mathbb{R}^{+}}}$, and
\(P=P_{\mathbb{R^{+}}}\) is the orthoprojector from \(L^2(\mathbb{R})\) onto
\(L^2(\mathbb{R}^{+})\), see~\eqref{NaPr}. For these objects, identities \eqref{Dilat} do not hold true for all \(n\in\mathbb{N}\), but only for \(n=1\). So, the operator \(\mathscr{F}\) is not a unitary dilation of its truncation~\(\mathscr{F}_{_{\scriptstyle\mathbb{R}^{+}}}\). Nevertheless, we succeeded in constructing a functional model of the operator~\(\mathscr{F}_{_{\scriptstyle\mathbb{R}^{+}}}\) such that it is easily analyzable. Analyzing this model, we can develop a complete spectral theory of the operator \(\mathscr{F}_{_{\scriptstyle\mathbb{R}^{+}}}\).

\section{The model space.}
\label{MS}
\begin{defn}{}\ \\
$1.$ The \emph{model space} \(\mathfrak{M}\)
is the set of all \(2\times1\) columns \(\varphi=
\displaystyle
\begin{bmatrix}
\varphi_{+}\\
\varphi_{-}
\end{bmatrix}
\) whose entries \(\varphi_{+}\) and
\(\varphi_{-}\) are arbitrary complex-valued functions of class~\(L^2(\mathbb{R^{+}})\). \\

\noindent $2.$ The space \(\mathfrak{M}\) is
equipped by the natural linear operations.

\noindent$3.$ The inner product \(\langle\varphi,\psi\rangle_{\mathfrak{M}}\) of columns
 \(\varphi=
\displaystyle
\begin{bmatrix}
\varphi_{+}\\
\varphi_{-}
\end{bmatrix}
\) and \(\psi=
\displaystyle
\begin{bmatrix}
\psi_{+}\\
\psi_{-}
\end{bmatrix}
\) belonging to this space
is defined as
\begin{equation}
\label{InP}
\langle\varphi,
\psi\rangle_{\mathfrak{M}}=\langle\varphi_{+},\psi_{+}\rangle_{L^2(\mathbb{R}^{+})}
+\langle\varphi_{-},\psi_{-}\rangle_{L^2(\mathbb{R}^{+})}.
\end{equation}
In particular,
\begin{equation}
\label{No}
\|\varphi\|^2_{\mathfrak{M}}=\|\varphi_{+}\|^2_{L^2(\mathbb{R}^{+})}+
\|\varphi_{-}\|^2_{L^2(\mathbb{R}^{+})}.
\end{equation}
\end{defn}

\begin{rem}
The model space \(\mathfrak{M}\) is merely the orthogonal sum of two copies
of the space \(L^2(\mathbb{R}^{+})\). The standard notation \(L^2(\mathbb{R^{+}})\oplus{}L^2(\mathbb{R^{+}})\) for such orthogonal sum
does not reflect the fact that the elements of \(\mathfrak{M}\) are $(2\times1)$-\emph{columns}. The notation
 \(\begin{bmatrix}
 L^2(\mathbb{R^{+}})\\{}\oplus\\{}L^2(\mathbb{R^{+}})
 \end{bmatrix}
 \) is more logical, but too bulky.
\end{rem}

\medskip

We define a linear mapping \(U\) of
the space \(L^2(\mathbb{R})\) into the model space.  For \(x\in{}L^2(\mathbb{R}^{+})\), the formal definition is
\begin{equation}
\label{FoD}
(Ux)(\mu)=
\begin{bmatrix}
(Ux)_{+}(\mu)\\[1.0ex]
(Ux)_{-}(\mu)
\end{bmatrix}
\ccomma \quad \mu\in\mathbb{R}^{+},
\end{equation}
where
\begin{subequations}
\label{fo}
\begin{align}
\label{fo+}
(Ux)_{+}(\mu)=&\frac{1}{\sqrt{2\pi}}\int\limits_{\mathbb{R}^{+}} x(\xi)\,\xi^{-1/2}\xi^{+i\mu}\,d\xi,\qquad \mu\in\mathbb{R}^+,\\
\label{fo-}
(Ux)_{-}(\mu)=&\frac{1}{\sqrt{2\pi}}\int\limits_{\mathbb{R}^{+}} x(\xi)\,\xi^{-1/2}\xi^{-i\mu}\,d\xi,\qquad \mu\in\mathbb{R}^+.
\end{align}
\end{subequations}
Here and in what follows, \(\xi^\zeta=e^{\zeta\ln\xi}\), where \(\ln\xi\in\mathbb{R}\) for \(\xi\in\mathbb{R}^{+}\).

If \(x\in{}L^2(R^{+})\), the functions \(x(\xi)\,\xi^{-1/2}\xi^{\pm{}i\mu}\) occurring
in \eqref{fo} may fail to be integrable with respect to  Lebesgue
measure \(d\xi\) on \(\mathbb{R}^{+}\). Therefore, the integrals
in \eqref{fo} may fail to exist as Lebesgue integrals.
\begin{defn}
\label{alf}
The \emph{set \(\mathcal{D}\)} is the set of all functions \(x\in{}L^2(\mathbb{R}^{+})\)  satisfying
\begin{equation}
\label{exc}
\int\limits_{\mathbb{R}^{+}}|x(\xi)|\xi^{-1/2}d\xi<\infty.
\end{equation}
\end{defn}

\begin{lem}
\label{ExU}
If a function \(x\) belongs to \(L^2(\mathbb{R}^{+})\)
and its support $\textup{supp}\,x$ lies  strictly
inside the positive semiaxis $\mathbb{R}^{+},$ then
$x\in\mathcal{D}$.
\end{lem}
\begin{proof}
\begin{equation*}
\begin{split}
\int\limits_{\mathbb{R}^{+}}
\big|x(\xi)|\xi^{-1/2}d\xi&=
\int\limits_{\xi\in\textup{supp}x}
\big|x(\xi)|\xi^{-1/2}\,d\xi
\\
&\leq
\bigg\{\int\limits_{\mathbb{R}^{+}}
|x(\xi)|^2d\xi\bigg\}^{1/2}\cdot\,\,
\bigg\{\int\limits_{\xi\in\textup{supp}x}
|\xi|^{-1}d\xi\bigg\}^{1/2}<\infty.\qedhere
\end{split}
\end{equation*}
\end{proof}
\noindent
For \(x\in\mathcal{D}\), the
integrals on the right-hand sides of \eqref{fo+} and \eqref{fo-}
exist as Lebesgue integrals for every $\mu\in\mathbb{R}^{+}$. So, the functions \((Ux)_{+}(\mu)\) and \break \((Ux)_{-}(\mu)\) are well defined
\emph{for every \(\mu\in\mathbb{R}^{+}\).}
\begin{lem}
\label{SqU}
If a function \(x\) belongs to \(\mathcal{D}\),
then both functions \((Ux)_{+}\) and \((Ux)_{-}\) belong to
\(L^2(\mathbb{R}^{+})\). Moreover, we have
\begin{equation}
\label{paE}
\|(Ux)_{+}\|_{L^2(\mathbb{R}^{+})}^{2}+\|(Ux)_{-}\|_{L^2(\mathbb{R}^{+})}^{2}=
\|x\|^2_{L^2(\mathbb{R}^{+})}.
\end{equation}
\end{lem}
\begin{proof}
 Changing the variable \(\xi=e^{\eta}\) in \eqref{fo},
we obtain
\begin{subequations}
\label{co}
\begin{align}
\label{co+}
(Ux)_{+}(\mu)=&\frac{1}{\sqrt{2\pi}}\int\limits_\mathbb{R} v(\eta)\,e^{+i\mu\eta}\,d\eta,\quad \mu\in\mathbb{R}^+,\\
\label{co-}
(Ux)_{-}(\mu)=&\frac{1}{\sqrt{2\pi}}\int\limits_\mathbb{R} v(\eta)\,e^{-i\mu\eta}\,d\eta,\quad \mu\in\mathbb{R}^+,
\end{align}
\end{subequations}
where
\begin{equation}
\label{dV}
v(\eta)=e^{\eta/2}x(e^{\eta}).
\end{equation}
We have
\begin{equation}
\label{cv}
\int\limits_{\mathbb{R}}|v(\eta)|^2d\eta=
\int\limits_{\mathbb{R}^{+}}|x(\xi)|^2d\xi.
\end{equation}
Put
\begin{equation}
\label{ca}
u(\nu)=
\begin{cases}
(Ux)_{+}(\,\,\nu\,\,)\ &\textup{if}\ \nu\in\mathbb{R}^{+},   \\[1.0ex]
(Ux)_{-}(-\nu)\ &\textup{if}\ \nu\in\mathbb{R}^{-}.
\end{cases}
\end{equation}
It is clear that
\begin{equation}
\label{ic}
\int\limits_{\mathbb{R}}|u(\nu)|^2d\nu=
\int\limits_{\mathbb{R}^{+}}|(Ux)_{+}(\mu)|^2d\mu+
\int\limits_{\mathbb{R}^{+}}|(Ux)_{-}(\mu)|^2d\mu.
\end{equation}
From \eqref{co} and \eqref{ca} it follows that
\begin{equation}
\label{os}
u(\nu)=\frac{1}{\sqrt{2\pi}}
\int\limits_{\mathbb{R}}v(\eta)\,e^{i\nu\eta}\,d\eta,\quad \nu\in\mathbb{R}.
\end{equation}
Thus,
\begin{equation}
\label{hfp}
u(\nu)=(\mathscr{F}v)(\nu),\quad \nu\in\mathbb{R},
\end{equation}
where \(\mathscr{F}\) is the Fourier--Plancherel operator \eqref{FWRA}. The Parceval itentity
\begin{equation}
\label{pe}
\int\limits_{\mathbb{R}}|u(\nu)|^2d\nu=
\int\limits_{\mathbb{R}}|v(\eta)|^2d\eta
\end{equation}
and formulas \eqref{cv} and \eqref{ic} yield~\eqref{paE}.
\end{proof}
It is clear that the set $\mathcal{D}$ is a (nonclosed) vector subspace
 of~\(L^2(\mathbb{R}^{+}\!)\).
\begin{lem}
\label{den}
The set $\mathcal{D}$ is dense in $L^2(\mathbb{R}^{+})$.
\end{lem}
\begin{proof}
Given \(x\in{}L^2(\mathbb{R}^{+})\), we define
\begin{equation}
\label{aps}
x_n(t)=x(t)\cdot\mathds{1}_{[1/n,n]}(t),\quad n=1,2,\dots,
\end{equation}
where \(\mathds{1}_{[1/n,n]}(t)\) is the indicator function of
the interval \([1/n,n]\). (See \eqref{IndF}.)
Clearly,
\begin{equation}
\|x-x_n\|_{L^2(\mathbb{R}^{+})}\to0\quad\textup{as}\quad{}n\to\infty.
\end{equation}
Moreover, \(x_n\in\mathcal{D}\) by Lemma \ref{ExU}.
\end{proof}
\begin{defn}
\label{DeU}
Formula \eqref{paE} means
that the operator~$U$ defined by\break \eqref{FoD}--\eqref{fo}  for \(f\in\mathcal{D}\)  maps
the subspace $\mathcal{D}$ into the model space \(\mathfrak{M}\) isometrically.
Therefore, the operator $U$  extends from
the subspace \(\mathcal{D}\subset{}L^2(\mathbb{R}^{+})\) to its closure \(L^2(\mathbb{R}^{+})\) by continuity:
\begin{equation}
\label{ext}
\begin{split}
\textup{if}\ x\in{}L^2(\mathbb{R}^{+}),\quad x_n\in\mathcal{D},\quad n=1,2,\dots,\quad x&=
\lim\limits_{n\to\infty}x_n, \\
\textup{then}\quad Ux
&\stackrel{\textup{def}}{=}
\lim\limits_{n\to\infty}Ux_n.
\end{split}
\end{equation}
We preserve the notation \(U\) for the
operator extended in this way.
\end{defn}

 \emph{From now on, we deal with the operator $U$ that is already extended from~$\mathcal{D}$ to the whole space $L^2(\mathbb{R}^{+})$ in accordance with~\eqref{ext}.}
It is clear that~\(U\) maps \(L^2(\mathbb{R}^{+})\) onto some closed subspace of the model space~\(\mathfrak{M}\).
 \begin{thm}
\label{on}
The operator \(U\) maps \(L^2(\mathbb{R}^{+})\) onto the whole model space~\(\mathfrak{M}\).
\end{thm}
\begin{proof}
Let \(y=
\displaystyle
\begin{bmatrix}
y_{+}\\
y_{-}
\end{bmatrix}\) be an arbitrary element of~$\mathfrak{M}$.
 Both functions \(y_{+}\) and \(y_{-}\) belong to \(L^2(\mathbb{R}^{+})\). We set
 \begin{equation}
 \label{auf}
 u(\nu)=
 \begin{cases}
 y_{+}(\,\nu\,)\quad &\textup{if}\quad \nu>0,\\
 y_{-}(-\nu)\quad &\textup{if}\quad \nu<0.
 \end{cases}
 \end{equation}
Clearly, $u\in{}L^2(\mathbb{R})$.  As a function of class~$L^2(\mathbb{R})$, the function~$u$ is representable in the form
\begin{equation}
\label{l2r}
u(\nu)=\frac{1}{\sqrt{2\pi}}\int\limits_{\mathbb{R}}v(\eta)e^{i\nu\eta}\,d\eta,\quad \nu\in\mathbb{R},
\end{equation}
where $v$ is a function in $L^2(\mathbb{R})$.

Relation~\eqref{l2r} can be interpreted as the following pair of formulas:
 \begin{subequations}
 \label{int}
 \begin{align}
  \label{int1}
 y_{+}(\mu)=&\frac{1}{\sqrt{2\pi}}\int\limits_{\mathbb{R}}
 v(\eta)e^{\,i\mu\eta\,}\,d\eta,\ \ \mu\in\mathbb{R^+},\\
  \label{int2}
 y_{-}(\mu)=&\frac{1}{\sqrt{2\pi}}\int\limits_{\mathbb{R}}
 v(\eta)e^{-i\mu\eta}\,d\eta,\ \ \mu\in\mathbb{R^+},
 \end{align}
 \end{subequations}
where \(v\in{}L^2(\mathbb{R})\). Changing the variable
\(\xi=e^{\eta}\) in \eqref{int}, we reduce \eqref{int}
to the form
 \begin{subequations}
 \label{rint}
 \begin{align}
  \label{rint1}
 y_{+}(\mu)=&\frac{1}{\sqrt{2\pi}}\int\limits_{\mathbb{R}^{+}}
 x(\xi)\xi^{-1/2}\xi^{\,i\mu\,}\,d\xi,\ \ \mu\in\mathbb{R^+},\\
  \label{rint2}
 y_{-}(\mu)=&\frac{1}{\sqrt{2\pi}}\int\limits_{\mathbb{R}^{+}}
 x(\xi)\xi^{-1/2}\xi^{-i\mu}\,d\xi,\ \ \mu\in\mathbb{R^+},
 \end{align}
 \end{subequations}
 where
 \begin{equation}
 \label{Cv}
 v(\eta)=e^{\eta/2}x(e^{\eta}).
 \end{equation}
 Moreover, \(x\in{}L^2(\mathbb{R}^{+})\):
 \begin{equation}
 \label{cvi}
 \int\limits_{\mathbb{R}^{+}}|x(\xi)|^2\,d\xi=
 \int\limits_{\mathbb{R}}|v(\eta)|^2\,d\eta.
 \end{equation}
By the definition of the operator \(U\),
formulas~\eqref{rint} mean that
 \begin{equation}
 \label{fin}
 y=Ux.
 \end{equation}
 \end{proof}
 \begin{rem}
 \label{Reg}
 The function \(v\) in \eqref{l2r} may fail to belong to \(L^1(\mathbb{R})\). To give a meaning to~\eqref{l2r}, we use the standard approximation procedure. We choose a sequence \(\{v_n\}_{n=1,2,\dots}\) such that
  $$
  v_n\in{}L^2(\mathbb{R})\cap{}L^1(\mathbb{R})
  $$
   for every \(n\) and
  \(\|v_n-v\|_{L^2(\mathbb{R})}\to0\) as \(n\to\infty\).
  The sequence \(\{u_n\}_{n=1,2,\dots}\), where
  $$
  u_n(t)=\frac{1}{\sqrt{2\pi}}
  \int\limits_{\mathbb{R}}v_n(\xi)e^{it\xi}\,d\xi,
  $$
   is well defined and converges to~\(u\):
  \(\|u_n-u\|_{L^2(\mathbb{R})}\to0\) as \(n\to\infty\).
 \end{rem}
 \begin{rem}
 \label{met}
 The transformation~\eqref{fo} \,is none other than the Mellin transform
 $
 \int_{\mathbb{R}^{+}}x(t)\,t^{\zeta-1}dt
 $
 restricted to the line
 $\textup{Re}\,\zeta=\frac{1}{2}:\zeta=\frac{1}{2}+i\mu$, $\mu\in\mathbb{R}$.  For\-mu\-la~\eqref{paE} is the Parceval identity for the Mellin transform, see
 \cite[\S3.17, Theorem 71]{T}. However, we would like to emphasize that we
 view this Mellin transform not as a single function defined for \(\mu\in\mathbb{R}\), but as a pair of functions defined for \(\mu\in\mathbb{R}^{+}\).
 \end{rem}
\section{The model of the truncated Fourier--Plancherel operator.}
\label{MFP}
In \S\ref{MS} we introduced the operator \(U\)
that maps the space \(L^2(\mathbb{R}^+)\) onto the model space
\(\mathfrak{M}\) isometrically.
In this section we calculate the operator
\(U\mathscr{F}_{\mathbb{R}^+}U^{-1}\),
which serves as a model of the operator  \(\mathscr{F}_{\mathbb{R}^+}\).

Let \(x\in{}L^2(\mathbb{R}^+)\) and
\(\displaystyle
\begin{bmatrix}
y_{+}\\
y_{-}
\end{bmatrix}
=Ux\), i.e., let~\eqref{fo} be true.
We would like to express the pair
\(\displaystyle
\begin{bmatrix}
z_{+}\\
z_{-}
\end{bmatrix}
=U\mathscr{F}_{\mathbb{R}^+}x\) in terms of the pair \(\displaystyle
\begin{bmatrix}
y_{+}\\
y_{-}
\end{bmatrix}\).
Substituting the function
$$
\big(\mathscr{F}_{\mathbb{R}^+}x\big)(t)=\frac{1}{\sqrt{2\pi}}
\int\limits_{\mathbb{R}^+}x(\xi)e^{it\xi}\,d\xi
$$
for~$x$ in~\eqref{fo}, we obtain
\begin{subequations}
 \label{sst}
 \begin{align}
  \label{sst1}
 z_{+}(\mu)=&\frac{1}{\sqrt{2\pi}}\int\limits_{\mathbb{R}^{+}}
 \bigg(\frac{1}{\sqrt{2\pi}}
\int\limits_{\mathbb{R}^+}x(\xi)e^{it\xi}\,d\xi\bigg)
t^{-1/2}t^{\,i{\mu}\,}\,dt,\quad \mu\in\mathbb{R^+},\\
  \label{sst2}
 z_{-}(\mu)=&\frac{1}{\sqrt{2\pi}}\int\limits_{\mathbb{R}^{+}}
 \bigg(\frac{1}{\sqrt{2\pi}}
\int\limits_{\mathbb{R}^+}x(\xi)e^{it\xi}\,d\xi\bigg)
t^{-1/2}t^{-i\mu}\,dt,\quad \mu\in\mathbb{R^+}.
 \end{align}
 \end{subequations}
 Changing the order of integration in \eqref{sst}, we arrive at the formulas
 \begin{subequations}
 \label{cst}
 \begin{align}
  \label{cst1}
 z_{+}(\mu)=&\frac{1}{\sqrt{2\pi}}\int\limits_{\mathbb{R}^{+}}x(\xi)
 \bigg(\frac{1}{\sqrt{2\pi}}
\int\limits_{\mathbb{R}^+}e^{it\xi}\,
t^{-1/2}t^{\,i\mu\,}\,dt\bigg)
 d\xi,\quad \mu\in\mathbb{R^+},\\
  \label{cst2}
 z_{-}(\mu)=&\frac{1}{\sqrt{2\pi}}\int\limits_{\mathbb{R}^{+}}x(\xi)
 \bigg(\frac{1}{\sqrt{2\pi}}
\int\limits_{\mathbb{R}^+}e^{it\xi}\,
t^{-1/2}t^{-i\mu}\,dt\bigg)
 d\xi,\quad \mu\in\mathbb{R^+}.
 \end{align}
 \end{subequations}
 Changing \(t\to{}t/\xi\) in \eqref{cst},
 we obtain
 \begin{subequations}
 \label{fst}
 \begin{align}
  \label{fst1}
 z_{+}(\mu)=&\frac{1}{\sqrt{2\pi}}\int\limits_{\mathbb{R}^{+}}x(\xi)\,
 \xi^{-1/2}\xi^{-i\mu}\bigg(\frac{1}{\sqrt{2\pi}}
\int\limits_{\mathbb{R}^+}e^{it}\,
t^{-1/2}t^{\,i\mu\,}\,dt\bigg)
 d\xi,\quad t\in\mathbb{R^+},\\
  \label{fst2}
 z_{-}(\mu)=&\frac{1}{\sqrt{2\pi}}\int\limits_{\mathbb{R}^{+}}x(\xi)\,
 \xi^{-1/2}\xi^{\,i\mu\,}\bigg(\frac{1}{\sqrt{2\pi}}
\int\limits_{\mathbb{R}^+}e^{it}\,
t^{-1/2}t^{-i\mu}\,dt\bigg)
 d\xi,\quad t\in\mathbb{R^+}.
 \end{align}
 \end{subequations}
 The inner integrals in \eqref{fst} do not depend on \(\xi\).
 Calculating these integrals, we can present~\eqref{fst} in the form
  \begin{subequations}
 \label{fft}
 \begin{align}
  \label{fft1}
 z_{+}(\mu)=&F_{+-}(\mu)\,y_{-}(\mu),\quad \mu\in\mathbb{R^+},\\
  \label{fft2}
 z_{-}(\mu)=&F_{-+}(\mu)\,y_{+}(\mu),\quad \mu\in\mathbb{R^+},
 \end{align}
 \end{subequations}
 where
 \begin{subequations}
 \label{ft}
 \begin{align}
\label{ft1}
 F_{+-}(\mu)=&\frac{1}{\sqrt{2\pi}}
\int\limits_{\mathbb{R}^+}e^{it}\,
t^{-1/2}t^{\,i\mu\,}\,dt,\quad
 \mu\in\mathbb{R^+}, \\
 \label{ft2}
 F_{-+}(\mu)=&\frac{1}{\sqrt{2\pi}}
\int\limits_{\mathbb{R}^+}e^{it}\,
t^{-1/2}t^{-i\mu}\,dt,\quad \mu\in\mathbb{R^+}.
 \end{align}
 \end{subequations}
 The functions \(F_{+-}\) and \(F_{-+}\) can be expressed in term of the Euler \(\Gamma\)-function:
 \begin{subequations}
 \label{Ft}
 \begin{align}
 \label{Ft1}
 F_{+-}(\mu)=&\frac{1}{\sqrt{2\pi}}\,e^{i\pi/4}
 e^{-\frac{\pi}{2}\mu}\,\Gamma\big(\tfrac{1}{2}+i\mu\big),\quad
 \mu\in\mathbb{R^+}, \\
  \label{Ft2}
 F_{-+}(\mu)=&\frac{1}{\sqrt{2\pi}}\,e^{i\pi/4}
 e^{\,\,\frac{\pi}{2}\mu}\,\,\Gamma\big(\tfrac{1}{2}-i\mu\big),\quad
 \mu\in\mathbb{R^+}.
 \end{align}
 \end{subequations}
  We shall not justify the possibility of changing the order of integration
 in \eqref{sst}.  The above argument, which leads from
 \eqref{sst} to \eqref{fft}, plays a heuristic role.
 Actually, we establish formulas \eqref{fft}, where the functions \(F_{+-}(\mu)\) and \(F_{-+}(\mu)\)
 are of the form \eqref{Ft}, in a different way.

 The pair of identities~\eqref{fft} can be presented in the
 matrix form
 \begin{equation}
 \label{meq}
 (U\mathscr{F}_{\mathbb{R}^{+}}x)(\mu)=
 F(\mu)(Ux)(\mu),\quad \mu\in\mathbb{R}^{+},
 \end{equation}
 where  \(F(\mu)\) is a $(2\times2)$-matrix:
  \begin{equation}
 \label{mF}
 F(\mu)=\begin{bmatrix}
0&F_{+-}(\mu)\\
F_{-+}(\mu)&0
\end{bmatrix}
\ccomma
\quad\mbox{for all }\,\mu\in\mathbb{\mathbb{R}^{+}}.
 \end{equation}
  \begin{thm}
 \label{MaTe}
 Let \(x\) be an arbitrary function in~\(L^2(\mathbb{R}^{+})\) and \(\mathscr{F}_{\mathbb{R}^{+}}x\) the truncated Fourier--Plancherel transform of \(x\). Then their images \(Ux\) and \(U\mathscr{F}_{\mathbb{R}^{+}}x\) under the operator \(U\) are related by formulas~\eqref{meq}, where the entries
  \(F_{+-}(\mu)\) and \(F_{-+}(\mu)\) of the matrix \(F(\mu)\)
 are of the form \eqref{Ft}.
 \end{thm}
 \begin{proof} It suffices to verify  identities \eqref{fft} only for \(x\) of the form
 \(x(t)=e_{a}(t)\), where
 \begin{equation}
 \label{ea}
 e_{a}(t)=e^{-at},\quad t\in\mathbb{R}^{+}
 \end{equation}
 and \(a\) is an arbitrary positive number. It is well known that the linear hull of the family of function \(\{e_a(t)\}_{0<a<\infty}\) is a dense set in \(L^2(\mathbb{R}^{+})\). The function \((\mathscr{F}_{\mathbb{R}^{+}}e_a)(t)\) can be calculated explicitly:
 \begin{equation}
 \label{fea}
 (\mathscr{F}_{\mathbb{R}^{+}}e_a)(t)=\frac{1}{\sqrt{2\pi}}\cdot\frac{1}{a-it}.
 \end{equation}
 The corresponding elements \(\displaystyle
\begin{bmatrix}
y_{+}\\
y_{-}
\end{bmatrix}
=Ue_a\) and
\(\displaystyle
\begin{bmatrix}
z_{+}\\
z_{-}
\end{bmatrix}
=(U\mathscr{F}_{\mathbb{R}^{+}})e_a\) can also be calculated explicitly. By the definition \eqref{FoD}--\eqref{fo} of the operator~$U$, we have
\begin{subequations}
 \label{uea}
 \begin{align}
  \label{uea1}
 y_{+}(\mu)=&a^{-\frac{1}{2}-i\mu}\,\Gamma\big(\tfrac{1}{2}+i\mu\big),\quad \mu\in\mathbb{R^+},\\
  \label{uea2}
 y_{-}(\mu)=&a^{-\frac{1}{2}+i\mu}\,\Gamma\big(\tfrac{1}{2}-i\mu\big),\quad \mu\in\mathbb{R^+},
 \end{align}
 \end{subequations}
 and
 \begin{subequations}
 \label{uFea}
 \begin{align}
  \label{uFea1}
 z_{+}(\mu)=&\sqrt{\frac{\pi}{2}}\,e^{i\frac{\pi}{4}}\,a^{-\frac{1}{2}+i\mu}\,
 \frac{e^{-\frac{\pi}{2}\mu}}{\cosh\pi\mu},\quad \mu\in\mathbb{R^+},\\
  \label{uFea2}
 z_{-}(\mu)=&\sqrt{\frac{\pi}{2}}\,e^{i\frac{\pi}{4}}\,a^{-\frac{1}{2}-i\mu}\,
\frac{e^{\frac{\pi}{2}\mu}}{\cosh\pi\mu}, \quad \mu\in\mathbb{R^+}.
 \end{align}
 \end{subequations}
Relations \eqref{fft} follow from \eqref{uea}, \eqref{uFea}, \eqref{Ft}, and the identity%
 \footnote{
 This is a special case of the identity \(\Gamma(\zeta)\cdot\Gamma(1-\zeta)=\frac{\pi}{\sin\pi\zeta},
 \quad\zeta\in\mathbb{C}\setminus\mathbb{Z}\).
 }
 \begin{equation}
\label{Refl}
\Gamma(1/2+{}i\mu)\,\Gamma(1/2-{}i\mu)=\frac{\pi}{\cosh{\pi\mu}}\,\cdot
\end{equation}
 \end{proof}
If  \begin{math}
 \label{ma}
 M=\begin{bmatrix}
M_{++}&M_{+-}\\
M_{-+}&M_{--}
\end{bmatrix}
 \end{math}
 is a $(2\times2)$-matrix with complex entries, then\break
 \(\|M\|_{\mathbb{C}^2\to\mathbb{C}^2}\) is the norm of the matrix \(M\) viewed as an operator in the space~\(\mathbb{C}^2\), where the space \(\mathbb{C}^2\) is equipped
 with the standard Hermitian norm.
 \begin{defn}
 Let \(M(\mu)=\begin{bmatrix}
M_{++}(\mu)&M_{+-}(\mu)\\
M_{-+}(\mu)&M_{--}(\mu)
\end{bmatrix}\)
be a $(2\times2)$-matrix-valued function of $\mu\in\mathbb{R}^{+}$ whose entries are complex valued functions defined almost everywhere. \emph{The multiplication
operator \(\mathcal{M}_M\) generated by the matrix function~\(M\)} is defined by the formula
\begin{equation}
\label{dmo}
(\mathcal{M}_My)(\mu)=M(\mu)y(\mu),\quad
y=\begin{bmatrix}
y_{+}\\
y_{-}
\end{bmatrix}
\in\,\mathfrak{M}.
\end{equation}
\end{defn}
\begin{lem}
\label{enmo}
If the matrix function \(M(\mu)\) is bounded on \(\mathbb{R}^{+}\), i.e.,
$$
\underset{\mu\in\mathbb{R}^{+}}{\textup{ess\,sup}}
\|M(\mu)\|_{\mathbb{C}^{2}\to\mathbb{C}^{2}}<\infty,
$$
then $\mathcal{M}_M$ is a bounded operator
in the space $\mathfrak{M}$, and
\begin{equation}
\label{eqn}
\|\mathcal{M}_M\|_{\mathfrak{M}\to\mathfrak{M}}=
\underset{\mu\in\mathbb{R}^{+}}{\textup{ess\,sup}}\,
\|M(\mu)\|_{\mathbb{C}^{2}\to\mathbb{C}^{2}}.
\end{equation}
\end{lem}
The expression on the left-hand side of \eqref{eqn} means
the norm of the multiplication operator \(\mathcal{M}_M\) in the space
\(\mathfrak{M}\).
Concerning the notion of $\textup{ess\,sup}$ see
\cite[\S2.11, p.~140.]{Bo}
\begin{rem}
\label{cont}
If the matrix function \(M(\mu)\) is continuous on \(\mathbb{R}^{+}\), then
\begin{equation}
\label{eqnC}
\|\mathcal{M}_M\|_{\mathfrak{M}\to\mathfrak{M}}=
\underset{\mu\in\mathbb{R}^{+}}{\textup{\,sup}}\,
\|M(\mu)\|_{\mathbb{C}^{2}\to\mathbb{C}^{2}}.
\end{equation}
\end{rem}

 Since \(x\in{}L^2(\mathbb{R}^{+})\) in \eqref{meq} is arbitrary,
 we can interpret~\eqref{meq} as an identity of operators. The following theorem is the core of the present
 paper.
 \begin{thm}
 \label{mt}
 The truncated Fourier--Plancherel operator \(\mathscr{F}_{\mathbb{R}^{+}}\) is unitarily equivalent to the multiplication operator \(\mathcal{M}_F\) generated by the matrix
 function~$F$ of the form \eqref{mF}--\eqref{Ft}
 in the space \(\mathfrak{M}\). We have
 \begin{equation}
 \label{UE}
\mathscr{F}_{\mathbb{R}^{+}}=U^{-1}\mathcal{M}_F\, U,
 \end{equation}
where \(U\) is the unitary operator described in Definition~{\rm\ref{DeU}}.
 \end{thm}
 \begin{rem}
 \label{mop}
 The multiplication operator \(\mathcal{M}_F\) possesses the same spectral properties as the operator~$\mathscr{F}_{\mathbb{R}^{+}}$. However, to study~$\mathcal{M}_F$ is much easier than
 to study~$\mathscr{F}_{\mathbb{R}^{+}}$. By the \emph{model
 operator} we mean~$\mathcal{M}_F$.
 \end{rem}
\section{The spectrum and the resolvent of the operator
\(\boldsymbol{\mathscr{F}_{\mathbb{R}^{+}}}\).
\label{SpARes}}
The unitary equivalence \eqref{UE} allows us to reduce the spectral analysis of
the operator \(\mathscr{F}_{\!_{\scriptstyle{\mathbb{R}^{+}}}}:\,L^2(\mathbb{R}^{+})\to{}L^2(\mathbb{R}^{+})\)
to that of the operator \(\mathcal{M}_{_{\scriptstyle F}}:\,\mathfrak{M}\to\mathfrak{M}\).

To perform the spectral analysis of the \emph{operator} \(\mathcal{M}_{_{\scriptstyle F}}\)
acting in the \emph{in\-fi\-ni\-te-dimen\-si\-o\-nal} space~\(\mathfrak{M}\),
 we need to perform the spectral analysis of the \break $(2\times2)$-mat\-rix \(F(\mu)\) acting in the \emph{two-dimensional} space \(\mathbb{C}^2\). The spectral analysis of the matrix \(F(\mu)\) can be done
for \emph{each} \(\mu\in\mathbb{R}^{+}\) separately. Then we can \emph{glue} the spectrum
\(\textup{\large\(\sigma\)}(\mathcal{M}_{F})\)  of~\(\mathcal{M}_{_{\scriptstyle F}}\) from the spectra \(\textup{\large\(\sigma\)}(F(\mu))\) of the matrices \(F(\mu)\),
 and  the resolvent of~\(\mathcal{M}_{_{\scriptstyle F}}\) can be glued from the resolvents of the matrices  \(F(\mu)\).

For \(\mu\in[0,\infty)\), let
\begin{equation}
\label{EiVa}
\zeta(\mu)=e^{i\pi/4}\frac{1}{\sqrt{2\,\cosh\pi\mu}},
\end{equation}
and let
\begin{equation}
\label{EiV}
\zeta_{+}(\mu)=\zeta(\mu),\quad \zeta_{-}(\mu)=-\zeta(\mu).
\end{equation}
It is clear that \(\zeta(\mu)\not=0\), so that \(\zeta_{+}(\mu)\not=\zeta_{-}(\mu)\) for every \(\mu\in[0,\infty)\).
\begin{lem}{ \ }
\label{Spmm}
For \(\mu\in[0,\infty)\), the spectrum
 \(\textup{\large\(\sigma\)}(F(\mu))\) of the matrix \(F(\mu)\) given by~\eqref{mF} is simple
and consists of two different points \(\zeta_{+}(\mu)\) and \(\zeta_{-}(\mu)\), see~\eqref{EiV} and~\eqref{EiVa}{\rm:}
\begin{equation}
\label{SpMM}
\textup{\large\(\sigma\)}(F(\mu))=\{\zeta_{+}(\mu)\,,\zeta_{-}(\mu)\}.
\end{equation}
\end{lem}
\begin{proof}
Let \(D(z,\mu)\) be the determinant of that matrix:
\begin{equation}
\label{detm}
D(z,\mu)=\det(zI-F(\mu)).
\end{equation}
The structure \eqref{mF} of~\(F(\mu)\) shows that
\begin{equation}
\label{di}
D(z,\mu)=z^2-F_{+-}(\mu)\cdot{}F_{-+}(\mu).
\end{equation}
The product \(F_{+-}(\mu)\cdot{}F_{-+}(\mu)\) can be calculated by using \eqref{Ft} and \eqref{Refl}:
$$
F_{+-}(\mu)~\cdot{}~F_{-+}(\mu)=\frac{i}{2\,\cosh\pi\mu}.
$$
 Thus,
\begin{equation}
\label{prdi}
D(z,\mu)=z^2-\frac{i}{2\,\cosh\pi\mu}.
\end{equation}
\end{proof}
\begin{defn}
\label{dein}
Let \(a\) and \(b\) be points in~\(\mathbb{C}\). By definition,
the interval \([a,\,b]\) is the set \([a,\,b]=\{(1-\tau)a+\tau{}b:\,\tau\,\,\textup{runs
over}\,\,[0,\,1]\}\). The open interval \((a,b)\) as well as
half-open intervals are defined similarly.
\end{defn}
When \(\mu\) runs over the interval \([0,\,\infty)\), the points \(\zeta_{+}(\mu)\) fill
the interval
\(\Big(0,\,\frac{1}{\sqrt{2}}\,e^{i\pi/4}\Big]\) and the points \(\zeta_{-}(\mu)\) fill
the interval \(\Big[-e^{i\pi/4}\frac{1}{\sqrt{2}},\,0\Big)\).
When \(\mu\) increases, the points \(\zeta_{+}(\mu)\), \(\zeta_{-}(\mu)\) move monotonically: the point \(\zeta_{+}(\mu)\) moves from \(e^{i\pi/4}\,\frac{1}{\sqrt{2}}\) to \(0\), the point \(\zeta_{-}(\mu)\) moves from \(-e^{i\pi/4}\,\frac{1}{\sqrt{2}}\) to \(0\).
 Thus, the mappings \(\mu\to\zeta_{+}(\mu)\) is a homeomorphism of \([0,\,\infty)\) onto
\(\Big(0,\,e^{i\pi/4}\,\frac{1}{\sqrt{2}}\Big]\) and the mapping \(\mu\to\zeta_{-}(\mu)\)
is a homeomorphism of \([0,\,\infty)\) onto \(\Big[-e^{i\pi/4}\frac{1}{\sqrt{2}},\,0\Big)\).
Moreover \(\zeta_{+}(\infty)=\zeta_{-}(\infty)=\{0\}.\)

Thus, the interval \(\Big[-e^{i\pi/4}\frac{1}{\sqrt{2}},\,e^{i\pi/4}\frac{1}{\sqrt{2}}\Big]\)
is naturally decomposed into the union of three nonintersecting parts:
\begin{equation}
\label{nade}
\Big[-e^{i\pi/4}\tfrac{1}{\sqrt{2}},\,e^{i\pi/4}\tfrac{1}{\sqrt{2}}\Big]=
\Big[-e^{i\pi/4}\tfrac{1}{\sqrt{2}},\,0\Big)\cup\{0\}\cup
\Big(0,\,e^{i\pi/4}\tfrac{1}{\sqrt{2}}\Big].
\end{equation}

\begin{thm}
\label{spmop}
The spectrum \(\textup{\large\(\sigma\)}(\mathcal{M}_{F})\)  of the model operator~$\mathcal{M}_{_{\scriptstyle F}}$ looks like this\/{\rm:}
\begin{equation}
\label{dsmo}
\textup{\large\(\sigma\)}(\mathcal{M}_{F})=
\Big[-e^{i\pi/4}\tfrac{1}{\sqrt{2}},\,e^{i\pi/4}\tfrac{1}{\sqrt{2}}\Big].
\end{equation}
\end{thm}
 In other words, Theorem \ref{spmop} claims that the spectrum \(\textup{\large\(\sigma\)}(\mathcal{M}_{F})\)  of~\(\mathcal{M}_{_{\scriptstyle F}}\) is represented in the form
\begin{equation}
\label{GlSp}
\textup{\large\(\sigma\)}(\mathcal{M}_{_{\scriptstyle F}})=
\bigcup_{\mu\in[0,\,\infty]}\textup{\large\(\sigma\)}(F(\mu)).
\end{equation}
Since the spectra of unitarily equivalent operators coincide, Theorem \ref{spmop} can be reformulated as follows.
\begin{thm}
\label{spfop}
For the spectrum \(\textup{\large\(\sigma\)}(\mathscr{F}_{\mathbb{R}^{+}})\)  of the truncated Fourier operator~\(\mathscr{F}_{\mathbb{R}^{+}}\) we have
\begin{equation}
\label{dsmot}
\textup{\large\(\sigma\)}(\mathscr{F}_{\mathbb{R}^{+}})=
\Big[-e^{i\pi/4}\tfrac{1}{\sqrt{2}},\,e^{i\pi/4}\tfrac{1}{\sqrt{2}}\Big].
\end{equation}
\end{thm}
In the present section we prove the description \eqref{GlSp}
of the spectrum \(\textup{\large\(\sigma\)}(\mathcal{M}_{_{\scriptstyle F}})\)
of the model operator \(\mathcal{M}_{F}\).
In passing, we obtain some estimates for the resolvents of the matrices
\(F(\mu)\). These estimates are not quite evident because the matrices
\(F(\mu)\) are not selfadjoint. In particular, \(F(\infty)\) is a
Jordan cell.

\begin{lem}
\label{EMME}
The norm of an arbitrary $(2\times2)$-matrix  \(M\),
\begin{equation*}
M=
\begin{bmatrix}
m_{11}&m_{12}\\[1.5ex]
m_{21}&m_{22}
\end{bmatrix}\,,
\end{equation*}
viewed as an operator from \(\mathbb{C}^2\) to \(\mathbb{C}^2\), admits the estimates
\begin{equation}
\label{EMMED}
\tfrac{1}{2}\,\textup{trace}\,(M^{\ast}M)
\leq\|M\|^2
\leq\,\textup{trace}\,(M^{\ast}M).
\end{equation}
Under the assumption that \(\det\,M\not=0\), the norm of the inverse matrix \(M^{-1}\) can
be estimated as follows\/{\rm:}
\begin{equation}
\label{EMMEI}
\begin{split}
|(\det\,M)|^{-2}\,\textup{trace}\,(M^{\ast}M)
&-\frac{2}{\textup{trace}\,(M^{\ast}M)}
\\
&\leq\|M^{-1}\|^{\,2}\leq
|(\det\,M)|^{-2}\,\textup{trace}\,(M^{\ast}M),
\end{split}
\end{equation}
where
\begin{equation}
\label{Trac}%
\textup{trace}\,M^{\ast}M=|m_{11}|^2+|m_{12}|^2+|m_{21}|^2+|m_{22}|^2.
\end{equation}
\end{lem}
\begin{proof}
Let \(s_0\) and \(s_1\) be the singular values of the matrix \(M\), i.e.,
\begin{equation}
\label{SinV}
0<s_1\leq{}s_0
\end{equation}
 and the numbers \(s_0^2,\,s_1^2\) are eigenvalues of the matrix
\(M^{\ast}M\). Then
\begin{gather*}
\|M\|=s_0,\quad \|M^{-1}\|=s_1^{\,-1},\\
 \textup{trace}(M^{\ast}M)=s_0^2+s_1^2,
\quad{}|\det(M)|^2=\det{}(M^{\ast}M)=s_0^2\cdot{}s_1^2.
\end{gather*}%
Therefore, inequality \eqref{EMMED} takes the form
\[\frac{1}{2}(s_{0}^{2}+s_{1}^{2})\leq{}s_{0}^{2}\leq(s_0^{2}+s_1^{2})\,,\]
and \eqref{EMMEI} takes the form
\[(s_0s_1)^{\,-2}(s_0^{\,2}+s_{1}^{\,2})\,-\,\frac{2}{s_0^{\,2}+s_1^{\,2}}\,\leq\,s_1^{\,-2}\,\leq\,
(s_0s_1)^{\,-2}(s_0^{\,2}+s_{1}^{\,2}).\]
The last inequalities are valid for arbitrary numbers \(s_0,\,s_1\) that satisfy~\eqref{SinV}.
\end{proof}
Since the numbers \(\Gamma(1/2\pm{}i\mu)\) are complex conjugate, from \eqref{Refl}
it follows that
\begin{equation}
\label{AbGa}
|\Gamma(1/2\pm{}i\mu)|^2=\frac{2\pi}{e^{\pi\mu}+e^{-\pi\mu}},\quad
\mu\in\mathbb{R}^{+}.
\end{equation}
Using \eqref{Ft} and \eqref{AbGa}, we calculate the absolute values \(|F_{+-}(\mu)|\) and \(|F_{-+}(\mu)|\):
 \begin{subequations}
 \label{av}
 \begin{align}
 \label{av1}
 |F_{+-}(\mu)|=&\frac{1}{\sqrt{1+e^{\,\,2\pi{}\mu}}}\,,\ \
 \mu\in\mathbb{R^+}, \\
  \label{av2}
 |F_{-+}(\mu)|=&\frac{1}{\sqrt{1+e^{-2\pi{}\mu}}},\ \
 \mu\in\mathbb{R^+}.
 \end{align}
 \end{subequations}
Note that, in particular,
\begin{equation}
\label{MaFC}
1/\sqrt{2} \leq|F_{-+}(\mu)|<1,\quad |F_{+-}(\mu)|\leq 1/\sqrt{2},\quad \mu\in\mathbb{R}^{+}.
\end{equation}
If \(\mu\) runs over the interval \([0,\infty)\), then
\(|F_{-+}(\mu)|\) increases from \(2^{-1/2}\) to \(1\) and
\(|F_{+-}(\mu)|\) decreases from \(2^{-1/2}\) to \(0\). In
particular,
\begin{equation}
\label{MaFCSu}
\sup_{\mu\in\mathbb{R}^{+}}|F_{-+}(\mu)|=\underset{\mu\in\mathbb{R}^{+}}%
{\textup{ess\,sup}}|F_{-+}(\mu)|=1.
\end{equation}
From \eqref{av} it follows that
\begin{equation}
\label{SMc2}
 |F_{+-}(\mu)|^2+|F_{-+}(\mu)|^2=1.
\end{equation}
 For the matrix \(F(t)\) defined by \eqref{mF}--\eqref{Ft} its norm is
 \begin{equation}
 \|F(t)\|_{\mathbb{C}^2\to\mathbb{C}^2}=\frac{1}{\sqrt{1+e^{-2\pi{}t}}}
 \quad\mbox{ for all }\ t\in\mathbb{R}^{+}.
 \end{equation}
\begin{lem}
\label{eremf}
For every \(\mu\in[0,\infty)\) and every \(z\in\mathbb{C}\setminus\!\textup{\large\(\sigma\)}(F(\mu))\),
the matrix \((zI-F(\mu))^{-1}\) admits the estimates
\begin{equation}
\begin{split}
\label{EsSRe}
 |D(z,\mu)|^{-2}\big(2|z|^2+1\big)&-\frac{2}{2|z|^2+1}
 \\
&\leq\|(zI-F(\mu))^{-1}\|^{2}\leq{}|D(z,\mu)|^{-2}\big(2|z|^2+1\big)\,,
\end{split}
\end{equation}%
where
 \begin{equation}
 \label{det}
 D(z,\mu)=\det(zI-F(\mu))
 \end{equation}
 and
\(\textup{\large\(\sigma\)}(F(\mu))\) is the spectrum of the matrix~\(F(\mu)\).
\end{lem}
\begin{proof} We apply estimate \eqref{EMMEI} to the matrix \(M=zI-F(\mu)\). We calculate the
quantity \(\textup{trace}\,M^{\ast}M\) with the help of~\eqref{Trac}:
\begin{equation}
\label{Etr}
\textup{trace}\,(zI-F(\mu))^{\ast}(zI-F(\mu))=2|z|^{2}+|F_{+-}(\mu)|^{2}+|F_{-+}(\mu)|^{2}.
\end{equation}
Using~\eqref{SMc2}, we see that
\begin{equation}
\textup{trace}\,(zI-F(\mu))^{\ast}(zI-F(\mu))=2|z|^{2}+1.
\end{equation}
\end{proof}
\begin{proof}[Proof of Theorem \ref{spmop}]
When \(\mu\) runs  over the interval \([0,\infty)\), the complex
numbers \(\dfrac{i}{2\cosh{\pi\mu}}\), which occur on the right-hand side of~\eqref{prdi}, fill the  interval~\((0,i/2]\). Therefore,
\begin{equation}
\label{DistCo}%
\inf_{\mu\in(0,\infty)}|D(z,\mu)|=\textup{dist}(z^2,\,[0,i/2]),
\end{equation}%
where
\begin{equation}
\label{Dist}%
\textup{dist}(z^2,\,[0,i/2]\,)=\min_{\zeta\in[0,i/2]}|z^2-\zeta|.
\end{equation}%
In particular,
\begin{equation*}
\Big(\inf_{\mu\in[0,\infty)}|D(z,\mu)|>0\Big)\Leftrightarrow
\big(\,z^2\not\in[0,i/2]\big),
\end{equation*}%
or, in other words,
\begin{equation}
\label{SepCo}
\Big(\inf_{\mu\in[0,\infty)}|D(z,\mu)|>0\Big)\Leftrightarrow
\Big(\,z\not\in\Big[-\frac{1}{\sqrt{2}}\,e^{i\pi/4},
\frac{1}{\sqrt{2}}\,e^{i\pi/4}\Big]\Big),
\end{equation}%
Inequalities \eqref{EsSRe} imply
\begin{equation}
\begin{split}
\label{CoFRe}
&\frac{2|z|^2+1}{\big(\inf_{\mu\in[0,\infty)}|D(z,\mu)|\big)^{2}}-\frac{2}{2|z|^2+1}
\\
&\quad\leq\sup\limits_{\mu\in[0,\infty)}\|(zI-F(\mu))^{-1}\|^{2}_{\mathbb{C}^2\to\mathbb{C}^2}\leq
\frac{2|z|^2+1}{\big(\inf_{\mu\in[0,\infty)}|D(z,\mu)|\big)^{2}}.
\end{split}
\end{equation}
From \eqref{SepCo} and \eqref{CoFRe} it follows that
\begin{equation}
\label{CoFRe1}
\Big(\sup\limits_{\mu\in[0,\infty)}\!\|(zI\!-\!F(\mu))^{-1}\|_{\mathbb{C}^2\to\mathbb{C}^2}\!<\!\infty\Big)\!\Leftrightarrow\!
\Big(z\!\not\in\!\Big[\!-\!\frac{1}{\sqrt{2}}\,e^{i\pi/4}.
\frac{1}{\sqrt{2}}\,e^{i\pi/4}\Big]\Big)
\end{equation}
\end{proof}
By Lemma \ref{enmo}, inequalities \eqref{CoFRe} can be viewred as estimates for the norm
  of the resolvent of the model operator \(\mathcal{M}_F\), or, in other words, as  estimates for the
  norm of the resolvent of the truncated Fourier--Plancherel operator~\(\mathscr{F}_{\mathbb{R}^{+}}\):
  \begin{equation}
  \begin{split}
\label{CoFReso}
\frac{2|z|^2+1}{\big(\textup{dist}(z^2,[0,i/2])\big)^{2}}-\frac{2}{2|z|^2+1}&\leq
\big\|\big(z\mathscr{I}-\mathscr{F}_{\mathbb{R}^{+}}
\big)^{-1}\big\|_{L^2(\mathbb{R}^{+})\to{}L^2(\mathbb{R}^{+})}^{2}
\\
&\leq
\frac{2|z|^2+1}{\big(\textup{dist}(z^2\,,\,[0,i/2]\,)\big)^{2}}\cdot
\end{split}
\end{equation}
The left inequality can be presented in the form
\begin{equation*}
\frac{(2|z|^2+1)^{1/2}}{\textup{dist\,}(z^2,[0,i/2]\,)}
\sqrt{1-\frac{2(\textup{dist\,}(z^2,\,[0,i/2]\,)^2}{(2|z|^2+1)^2}}
\leq\big\|\big(z\mathscr{I}-\mathscr{F}_{\mathbb{R}^{+}}
\big)^{-1}\big\|_{L^2(\mathbb{R}^{+})\to{}L^2(\mathbb{R}^{+})}.
\end{equation*}
Here, the quantity under the square root  is positive
because
\begin{equation*}
\label{GOH}%
\frac{2\textup{dist\, }^2(z^2,\,[0,i/2])}{\big(2|z|^2+1\big)^2}\leq
\frac{2|z|^2}{\big(2|z|^2+1\big)^2}\leq\frac{1}{2}.
\end{equation*}
Since \((1-\alpha)\leq\sqrt{1-\alpha}\) for \(0\leq\alpha\leq1\), we have
\[1-\frac{2\textup{dist}^2(z^2,\,[0,i/2])}{\big(2|z|^2+1\big)^2}
\leq\sqrt{1-\frac{2\textup{dist}^2(z^2,\,[0,i/2])}{\big(2|z|^2+1\big)^2}}
.\]
Thus, we get a lower estimate for the norm of the resolvent:
\begin{subequations}
\label{EsRes}
\begin{equation}
\label{LoEs}
\frac{\big(2|z|^2+1\big)^{1/2}}{\textup{dist}(z^2,[0,i/2])}-
\frac{2\textup{dist}(z^2,[0,i/2])}{\big(2|z|^2+1\big)^{3/2}}
\leq\big\|(z\mathscr{I}-\mathscr{F}_{\mathbb{R}^{+}})^{-1}\big\|_{L^2(\mathbb{R}^{+})\to{}L^2(\mathbb{R}^{+})}.
\end{equation}
An upper estimate for the norm of the resolvent is provided by the right inequality in~\eqref{CoFReso}:
\begin{equation}
\label{UpEs}
\big\|(z\mathscr{I}-\mathscr{F}_{\mathbb{R}^{+}})^{-1}\big\|_{L^2(\mathbb{R}^{+})\to{}L^2(\mathbb{R}^{+})}
\leq\frac{\big(2|z|^2+1\big)^{1/2}}{\textup{dist}(z^2\,,\,[0,i/2])}\cdot
\end{equation}
\end{subequations}
The smaller is the value $\dist(z^2,[0,i/2])$, the closer are the lower estimate \eqref{LoEs}
and the upper estimate \eqref{UpEs}.

 However, we would like to estimate  the resolvent not in terms of
 $$
 \dist(z^2,\,[0,i/2]),
 $$
  but rather in terms of
  $
  \dist(z,\,\textup{\large\(\sigma\)}(\mathscr{F}_{\mathbb{R}^{+}})).
  $
\begin{lem}
\label{DfDi}
Let \(\zeta\) be a point of the spectrum \(\textup{\large\(\sigma\)}(\mathscr{F}_{\mathbb{R}^{+}}))\)\textup{:}
\begin{equation}
\label{zePsp}
\zeta\in\bigg[-\frac{1}{\sqrt{2}}\,e^{i\pi/4},\,
\frac{1}{\sqrt{2}}\,e^{i\pi/4}\bigg]\,,
\end{equation}
and let $z$  lie on the normal to the interval \(\Big[-\frac{1}{\sqrt{2}}\,e^{i\pi/4},\, \frac{1}{\sqrt{2}}\,e^{i\pi/4}\Big]\) at the point~$\zeta\!:$
\begin{equation}
\label{znorspm}
z=\zeta\pm{}|z-\zeta|e^{i3\pi/4}.
\end{equation}
Then
\begin{equation}
\label{distot}
\textup{dist}\big(z^2\,,\big[0\,,i/2\big]\,\big)\,=\,
\begin{cases}
2|\zeta|\,|z-\zeta|\,&\textup{if \ }|z-\zeta|\leq|\zeta|\,,\\
|\zeta|^2+|z-\zeta|^2\,=\,|z|^2&\textup{if \ }|z-\zeta|\geq|\zeta|.
\end{cases}
\end{equation}
\end{lem}
\begin{proof}
Condition \eqref{zePsp} means that
\(\zeta=\pm|\zeta|e^{i\pi/4}\). Substituting this  in~\eqref{znorspm}, we obtain
\begin{equation*}
z^2=\pm2|\zeta|\,|z-\zeta|+i(|\zeta|^2-|z-\zeta|^2).
\end{equation*}
If \(|z-\zeta|\leq|\zeta|\), then the point \(i(|\zeta|^2-|z-\zeta|^2)\)
lies on the interval \([0,i/2]\). In this case,
$$
\dist\,(z^2,\,[0,i/2])=2|\zeta|\,|z-\zeta|.
$$
 If \(|z-\zeta|\geq|\zeta|\),
then the point \(i(|\zeta|^2-|z-\zeta|^2)\) lies on the half-axis \([0,-i\infty)\).
In this case,
$$
\textup{dist}\,\big(z^2,[0,\,i/2]\big)=\sqrt{{\big(|\zeta|^2-|z-\zeta|^2\big)^2+
4|\zeta|^2|z-\zeta|^2}}=|\zeta|^2+|z-\zeta|^2=|z|^2.
$$
Since \(|\zeta|^2+|z-\zeta|^2\geq2|\zeta||z-\zeta|\), it follows that, in any of the cases above,  \(|z-\zeta|\leq|\zeta|\), or \(|z-\zeta|\geq|\zeta|\), we have
\begin{equation}
\label{ERFB}
\dist\,(z^2\,,\,[0,\,i/2]))\geq2|\zeta|\,|z-\zeta|.
\end{equation}
holds.
\end{proof}
\begin{thm}%
\label{TEstRes}%
Let \(\zeta\) be a point of the spectrum \(\textup{\large\(\sigma\)}(\mathscr{F}_{\mathbb{R}^{+}})\)
of the operator \(\mathscr{F}_{\mathbb{R}^{+}}\), and let~\(z\) lie
on the normal to the interval \(\textup{\large\(\sigma\)}(\mathscr{F}_{\mathbb{R}^{+}})\)
at the point \(\zeta\).\\
Then the following statements hold true.
\begin{enumerate}
\item[\textup{1.}]
The resolvent \((z\mathscr{I}-\mathscr{F}_{\mathbb{R}^{+}})^{-1}\) admits the upper estimate
\begin{equation}
\label{UpEsResSM}
\big\|(z\mathscr{I}-\mathscr{F}_{\mathbb{R}^{+}})^{-1}\big\|
_{L^2(\mathbb{R}^{+})\to{}L^2(\mathbb{R}^{+})}\leq
A(z)\frac{1}{|\zeta|}\cdot\frac{1}{|z-\zeta|}\,,
\end{equation}
where
\begin{equation}
\label{Az}
A(z)=\frac{(2|z|^2+1)^{1/2}}{2}\cdot
\end{equation}
\item[\textup{2.}]
If, moreover, the condition \(|z-\zeta|\leq|\zeta|\) is satisfied,
then the resolvent \((z\mathscr{I}-\mathscr{F}_{\mathbb{R}^{+}})^{-1}\) also admits the lower estimate
\begin{equation}
\label{LoEsResSM}
\begin{split}
\hspace{3.0ex}A(z)\frac{1}{|\zeta|}\cdot\frac{1}{|z-\zeta|}
&-B(z)|\zeta||z-\zeta|
\\
&\leq\big\|(z\mathscr{I}-\mathscr{F}_{\mathbb{R}^{+}})^{-1}
\big\|_{L^2(\mathbb{R}^{+})\to{}L^2(\mathbb{R}^{+})}\,\ccomma
\end{split}
\end{equation}
 where \(A(z)\) is the same as in \eqref{Az} and
\begin{equation}
\label{Bz}
B(z)=\frac{4}{(2|z|^2+1)^{3/2}}\,\cdot
\end{equation}
\item[\textup{3.}] For \(\zeta=0\), the resolvent \((z\mathscr{I}-\mathscr{F}_{\mathbb{R}^{+}})^{-1}\)
admits the estimates
\begin{align}
\label{EZEZM}
\hspace{3.0ex}2A(z)\frac{1}{|z|^2}-B(z)
\leq\big\|(z\mathscr{I}-\mathscr{F}_{\mathbb{R}^{+}})^{-1}
\big\|_{L^2(\mathbb{R}^{+})\to{}L^2(\mathbb{R}^{+})}
\leq{}2A(z)\frac{1}{|z|^2},
\end{align}
where \(A(z)\) and \(B(z)\) are the same as in \eqref{Az},
 \eqref{Bz}, and \(z\) is an arbitrary point of the normal.
\end{enumerate}
In particular, if \(\zeta\not=0\), and \(z\) tends to \(\zeta\)
along the normal to the interval \(\textup{\large\(\sigma\)}(\mathscr{F}_{\mathbb{R}^{+}})\),
then we have
\begin{equation}
\label{AsNzpM}
\big\|(z\mathscr{I}-\mathscr{F}_{\mathbb{R}^{+}})^{-1}\big\|_{L^2(\mathbb{R}^{+})
\to{}L^2(\mathbb{R}^{+})}=\frac{A(\zeta)}{|\zeta|}\,\frac{1}{|z-\zeta|}
+O(1).
\end{equation}
If \(\zeta=0\) and \(z\) tends to \(\zeta\)
along the normal to the interval \(\textup{\large\(\sigma\)}(\mathscr{F}_{\mathbb{R}^{+}})\),
then we have
\begin{equation}
\label{AszpM}
\big\|(z\mathscr{I}-\mathscr{F}_{\mathbb{R}^{+}})^{-1}\big\|
_{L^2(\mathbb{R}^{+})\to{}L^2(\mathbb{R}^{+})}=|z|^{-2}+O(1)\,,
\end{equation}
where \(O(1)\) is a quantity that remains bounded as \(z\) tends to \(\zeta\).
\end{thm}%
\begin{proof}
The proof is based on  estimates \eqref{EsRes} for the resolvent
and on Lem\-ma~\ref{DfDi}. Combining inequality~\eqref{ERFB}
with estimate~\eqref{UpEs}, we obtain estimate~\eqref{UpEsResSM},
which is valid for \emph{all} \(z\) lying on the normal to
the interval \(\textup{\large\(\sigma\)}(\mathscr{F}_{\mathbb{R}^{+}})\) at the point~$\zeta$.
\emph{If, moreover,  \(z\) is suffisiently close to \(\zeta\)}, namely, the condition \(|z-\zeta|\leq|\zeta|\)
is satisfied, then equality occurs in~\eqref{ERFB}. Combining \emph{identity}~\eqref{ERFB} with estimate~\eqref{LoEs}, we obtain~\eqref{LoEsResSM}.

The asymptotic relation \eqref{AsNzpM} is a consequence of  inequalities
\eqref{LoEsResSM} and \eqref{EZEZM} because
$$
\frac{|A(z)-A(\zeta)|}{|z-\zeta|}=O(1)
$$
as \(z\) tends to \(\zeta\).

The asymptotic relation \eqref{AszpM} is a consequence of inequalities
\eqref{EsRes} and the relation
$$
\textup{dist}\,(z^2,\,[0,\,i/2])=|z|^2,
$$
which is valid for all \(z\) lying on the normal to
the interval \(\textup{\large\(\sigma\)}(\mathcal{M}_{F})\) at the point \(\zeta=0\).
(See \eqref{distot} for \(\zeta=0\).)
\end{proof}
\begin{cor}
\label{NoSNO}
 The operator \(\mathscr{F}_{\mathbb{R}^{+}}\) is not similar to any normal operator.
\end{cor}
Should the operator \(\mathscr{F}_{\mathbb{R}^{+}}\) be similar to some
normal operator \(\mathcal{N}\),  the resolvent of the operator \(\mathscr{F}_{\mathbb{R}^{+}}\) would admit the estimate
\[\big\|(z\mathscr{I}-\mathscr{F}_{\mathbb{R}^{+}})^{-1}\big\|_{L^2(\mathbb{R}^{+})\to{}L^2(\mathbb{R}^{+})}%
\leq\,C\,(\textup{dist}\,(z,\,\textup{\large\(\sigma\)}(\mathscr{F}_{\mathbb{R}^{+}}))^{-1}\,,\]
with \(C<\infty\) independent of~\(z\). However,
this estimate  is not compatible with the asymptotic relation \eqref{AszpM}.

\end{document}